\documentclass[nohistory]{degruyter-journal-a}           

\title{Convergence analysis of a variational quasi-reversibility approach for an inverse hyperbolic heat conduction problem}
\headlinetitle{Quasi-reversibility method for an inverse hyperbolic problem}

\lastnameone{Anh Khoa}
\firstnameone{Vo}
\nameshortone{V.\,A.~Khoa}
\addressone{Department of Mathematics and Statistics, University of North Carolina at Charlotte, Charlotte, North Carolina 28223}
\countryone{USA}
\emailone{vakhoa.hcmus@gmail.com, avo5@uncc.edu}

\lastnametwo{Dao}
\firstnametwo{Manh-Khang}
\nameshorttwo{M.\,-K.~Dao}
\addresstwo{Department of Mathematics, KTH Royal Institute of Technology, 100 44 Stockholm}
\countrytwo{Sweden}
\emailtwo{mkdao@kth.se}

\abstract{ We study a time-reversed hyperbolic heat conduction problem based upon the Maxwell--Cattaneo model of non-Fourier heat law. This heat and mass diffusion problem is a hyperbolic type equation for thermodynamics systems with thermal memory or with finite time-delayed heat flux, where the Fourier or Fick law is proven to be unsuccessful with experimental data. In this work, we show that our recent variational quasi-reversibility method for the classical time-reversed heat conduction problem, which obeys the Fourier or Fick law, can be adapted to cope with this hyperbolic scenario. We establish a generic regularization scheme in the sense that we perturb both spatial operators involved in the PDE. Driven by a Carleman weight function, we exploit the natural energy method to prove the well-posedness of this regularized scheme. Moreover, we prove the H\"older rate of convergence in the mixed  $L^2$--$H^1$ spaces.}



\researchsupported{This work was supported by US Army Research Laboratory and US Army Research Office grant W911NF-19-1-0044. The work of V.A.K was also partly supported by the Research Foundation-Flanders (FWO) in Belgium under the project named ``Approximations for forward and inverse reaction-diffusion problems related to cancer models". The work of M.-K.D was supported by the Swedish Research Council grant (2016-04086).}

\acknowledgments{V.A.K would like to thank Prof. Michael Victor Klibanov (Charlotte, USA) for his wholehearted guidance during the fellowship at UNCC and for giving a chance to delve into the Carleman estimates and convexification.}

\dedicated{Dedicated to Professor Michael Victor Klibanov on his 70th birth anniversary}

\begin{document}

\section{Introduction}
\subsection{Statement of the inverse problem}
In this work, we are interested in the extension of our new  quasi-reversibility (QR) method in \cite{Nguyen2019} for terminal boundary value problems. In this regard, we want to recover the initial distribution of an evolutionary equation, given the terminal data. This model is well known to be one of the classical problems in the field of inverse and ill-posed problems; cf. e.g. \cite{Kabanikhin2008} for some background of typical models in this research line. As to the applications of this model, having a reliable stable approximation of this backward-in-time problem is significantly helpful in many physical, biological and ecological contexts. Those are concretely involved in, e.g., the works \cite{Tuan2018,Carasso1978,Jaroudi2016,Sambatti2019}. In particular, the first contribution of this model being in mind relies on the heating/cooling transfer problem based upon the fact that sometimes, we want to measure the initial temperature of a material and our equipment only works at a given later time. Recently, this scenario has been extended to the case of a two-slab composite system with an ideal transmission condition in \cite{Tuan2018}. The second application we would like to address here is recovering blurry digital images acquired by camera sensors. This practical concern was initiated in \cite{Carasso1978} and has been scrutinized in the framework of source localization for brain tumor in \cite{Jaroudi2016}. In mathematical oncology, reconstructing the initial images of the tumor can be used for analyzing behaviors of cancer cells and then potentially for predicting the progression
of neoplasms of early-stage patients. This initial reconstruction is also part of the so-called data assimilation procedure that has been of interest so far in weather forecasting (cf. \cite{Asch2016}).

It is worth mentioning that considerations of such parabolic models indicate the use of the Fourier or Fick law. However, in some contexts of thermodynamics this typical law is proven to be unsuccessful with experimental data. In fact, any
initial disturbance in a medium is propagated instantly when taking into account the parabolic case; cf. e.g. \cite{Christov2005}. We also refer to the monograph \cite{Jou2001} and some impressive works \cite{Vadasz2005,Mendez1997}, where some electromechanical models were studied to unveil this non-standard incompatibility. In order to avoid the phenomenon of infinite
propagation, the Cattaneo--Vernotte law was derived, proposing that the parabolic case should be upgraded to a hyperbolic form. In terms of PDEs, it means one should consider
\begin{align}
	u_{tt} + u_{t} - \Delta u = 0\quad\text{in }\Omega\times\left(0,T\right), \label{11}
\end{align}
where $T>0$ is the final time and $\Omega\subset\mathbb{R}^{d}$ ($d=1,2,3$) is a regular bounded domain of interest with a sufficiently smooth boundary. In electrodynamics, equation (\ref{11}) is the same as the telegrapher's equation derived from the Maxwell equation. That is why one usually refers (\ref{11}) to as the Maxwell--Cattaneo model.

In this work, we investigate a generalized model of (\ref{11}) due to our mathematical interest. We assume to look for $u(x,t):\Omega\times (0,T)\to \mathbb{R}$ satisfying the following evolutionary equation:
\begin{equation}
u_{tt}+u_{t}-\Delta u-\Delta u_{t}=0\quad\text{in }\Omega\times\left(0,T\right).\label{eq:ori1}
\end{equation}
In the studies of the motion of viscoelastic materials, this is well-known to be the linear strongly damped wave equation, where the weak and strong damping terms ($u_t$ and $-\Delta u_t$) are altogether involved in the PDE. Cf. \cite{Bonetti2017} and references cited therein, the solution $u$ in that setting can be viewed as a displacement, whilst it is a temperature field in the context of thermodynamics we have mentioned above.  
Going back to the heat context, we note that the underlying equation (\ref{eq:ori1}) is also related to the so-called Gurtin--Pipkin model, which reads as
\begin{align}
	\theta_t = \int_{0}^{t} \kappa(t-s)\theta_{xx}(s)ds.\label{22}
\end{align}
When the kernel $\kappa$ is a constant, (\ref{22}) becomes an integrated wave equation after differentiation in time. If $\kappa(t) = e^{-t}$, one has the weakly damped wave equation $u_{tt} + u_{t} - u_{xx} = 0$. Furthermore, when $\kappa(t) = \delta(t)$, we get back to the classical heat equation. Therefore, we can conclude that our mathematical analysis for (\ref{eq:ori1}) really works for many distinctive physical applications at the same time.

To complete the time-reversed model, we endow (\ref{eq:ori1}) with the following boundary and terminal conditions:
\begin{equation}
\begin{cases}
u\left(x,t\right)=0\quad\text{ on }\partial\Omega\times\left(0,T\right),\\
u\left(x,T\right)=f_{0}\left(x\right),u_{t}\left(x,T\right)=f_{1}\left(x\right)\text{ in }\Omega.
\end{cases}\label{eq:ori2}
\end{equation}
Hence, (\ref{eq:ori1}) and (\ref{eq:ori2}) form our terminal boundary value problem. As to the ill-posedness of this problem, we refer to \cite{Tuan2017b} for proof of its natural instability using the spectral approach.

\subsection{Historical remarks and contributions of the paper}
In the context of the time-reversed parabolic problem, many regularization schemes were extensively designed in order to circumvent its natural ill-posedness. Inverse problems for parabolic equations with memory effects were investigated in \cite{Avdonin2019,Luan2019}. Since the aim of this work is extending our new QR method in \cite{Nguyen2019} to the hyperbolic heat conduction scenario, we would like to address some existing literature just on the QR topic close to the explicit technique we are developing. Meanwhile, some implicit QR methods for the backward heat conduction problem can be referred to the works \cite{Ewing1975,Long1994,Long1996,Doan2017}. The ``implicit'' here means that the scheme is designed by perturbing the kernel of the unbounded operator itself. Another QR-based approaches using minimization were studied in e.g. \cite{Klibanov2015,Klibanov2019a}.

The very first idea about quasi-reversibility of time-reversed parabolic problems was established by Latt\`es and Lions in the monograph \cite{Lattees1967} when they used a fourth-order spatial perturbation to stabilize the Laplace operator involved in the classical time-reversed parabolic equation. Motivated by this approach, several modifications and variants were constructed and analyzed through five decades, which makes this method considerable in the field of inverse and ill-posed problems. For example, we mention here the pioneering work \cite{Showalter1970}, where a third-order operator in space and time was proposed to obtain a regularization scheme in the form of a pseudoparabolic equation. Recently, Kaltenbacher et al. \cite{Kaltenbacher2019} has used a nonlocal perturbing operator in time with fractional order to regularize the ill-posed problem.

Our newly developed QR method follows the original idea of Latt\`es and Lions, i.e. we only use the spatial perturbation to stabilize the unbounded spatial operator. The key ingredient of our method lies in the fact that we use the perturbation operator to turn the inverse problem into a forward-like problem involving the stabilized operator. This notion has been studied in a spectral form in our recent work \cite{Tuan2017a}. As a follow-up, we generalize this method in \cite{Nguyen2019} by the establishment of conditional estimates for both the perturbation and stabilized operators. Driven by a Carleman weight function, we further apply the conventional energy method to show both well-posedness of the regularized system and error bounds. This way allows us to derive the scheme in the finite element setting and prove the error estimates in the finite-dimensional space. This will be our next target work in the future.

This work is the first time we extend our new method to the ill-posed problem (\ref{eq:ori1}) and (\ref{eq:ori2}). Intuitively, we construct in section \ref{sec:2-1} a generic regularized system in the sense that we perturb all the spatial terms $-\Delta u$ and $-\Delta u_t$. We then use the conditional estimates established in \cite{Nguyen2019} to obtain the H\"older rate of convergence in section \ref{sec:4}. Besides, well-posedness of the regularized system is considered in section \ref{sec:2} using a priori estimates and compactness arguments. 

\section{A variational quasi-reversibility framework} \label{sec:2-1}
To this end, $\left\langle \cdot,\cdot\right\rangle $ indicates either the scalar product in $L^2(\Omega)$ or the dual pairing
of a continuous linear functional and an element of a function space. Also, $\left\Vert \cdot\right\Vert $ is the norm in $L^2(\Omega)$. Different inner products and norms should be written as $\left\langle \cdot,\cdot\right\rangle_{X_1} $ and $\left\Vert \cdot\right\Vert_{X_2} $, respectively, where $X_1$ is a certain Hilbert space and $X_2$ is a Banach space. In the sequel, we denote $\varepsilon\in(0,1)$ by the noise level of the terminal data $f_0,f_1$ in (\ref{eq:ori2}). Any constant $C>0$ may vary from line to line. We usually indicate its dependencies if necessary.

We introduce an auxiliary function $\gamma:=\gamma(\varepsilon)\ge 1$ satisfying $\lim_{\varepsilon\to 0}\gamma(\varepsilon) = \infty$.

\begin{definition}[perturbing operator]\label{def1}
	The linear mapping $\mathbf{Q}_{\varepsilon}:L^2(\Omega)\to L^2(\Omega)$ is said to be a family of $\varepsilon$-dependent perturbing operator if there exist a function space $\mathbb{W}\subset L^2(\Omega)$ and a noise-independent constant $C_0 > 0$ such that
	\begin{align}\label{QQ}
		\left\Vert \mathbf{Q}_{\varepsilon}u\right\Vert \le C_{0}\left\Vert u\right\Vert _{\mathbb{W}}/\gamma(\varepsilon)\quad\text{for any }u\in\mathbb{W}.
	\end{align}
\end{definition}

\begin{definition}[stabilized operator]\label{def2}
	The linear mapping $\mathbf{P}_{\varepsilon}: L^2(\Omega)\to L^2(\Omega)$ is said to be a family of $\varepsilon$-dependent stabilized operator if there exists a noise-independent constant $C_1 > 0$ such that
	\begin{align}\label{PP}
		\left\Vert \mathbf{P}_{\varepsilon}u\right\Vert \le C_{1}\log\left(\gamma\left(\varepsilon\right)\right)\left\Vert u\right\Vert \quad\text{for any }u\in L^{2}\left(\Omega\right).
	\end{align}
\end{definition}

In this work, we start off with the generic approach of this modified version by stabilizing both two terms $-\Delta u$ and $-\Delta u_t$. By choosing the stabilization $\mathbf{P}_{\varepsilon}=2\Delta+\mathbf{Q}_{\varepsilon}$, our regularized equation is of the following form:
\begin{equation}
u_{tt}^{\varepsilon}+u_{t}^{\varepsilon}+\Delta u^{\varepsilon}+\Delta u_{t}^{\varepsilon}=\mathbf{P}_{\varepsilon}u^{\varepsilon}+\mathbf{P}_{\varepsilon}u_{t}^{\varepsilon}\quad\text{in }\Omega\times\left(0,T\right).\label{eq:regu1}
\end{equation}
It is worth noting that these perturbing and stabilized operators are constructed with respect to the variable $x$ only. Our main purpose in this finding is that we are able to obtain the convergence analysis of a family of regularization schemes based upon some particular conditional estimates of such perturbations and stabilizations. To qualify the convergence of such regularization schemes, the conditional estimates (\ref{QQ}) and (\ref{PP}) are particularly needed.  Some particular choices of these perturbing and stabilized operators shall be discussed in Remark \ref{rem2}.

Now we complete our regularized problem. Since in real-world applications the terminal data are usually noisy, we 
endow (\ref{eq:regu1}) with the following boundary and terminal conditions:
\begin{equation}
\begin{cases}
u^{\varepsilon}\left(x,t\right)=0\quad \text{ on }\partial\Omega\times\left(0,T\right),\\
u^{\varepsilon}\left(x,T\right)=f_{0}^{\varepsilon}\left(x\right),u_{t}^{\varepsilon}\left(x,T\right)=f_{1}^{\varepsilon}\left(x\right)\text{ in }\Omega.
\end{cases}\label{eq:regu2}
\end{equation}
In (\ref{eq:regu2}), we assume to have a noise level $\varepsilon\in (0,1)$ such that
\begin{equation}
\left\Vert u^{\varepsilon}\left(\cdot,T\right)-u\left(\cdot,T\right)\right\Vert _{H^{1}\left(\Omega\right)}+\left\Vert u_{t}^{\varepsilon}\left(\cdot,T\right)-u_{t}\left(\cdot,T\right)\right\Vert \le\varepsilon.\label{eq:measure}
\end{equation}
To validate our mathematical analysis below, we suppose that $f_0,f_0^{\varepsilon}\in H^1(\Omega)$ and $f_1,f_1^{\varepsilon}\in L^2(\Omega)$.

\begin{remark}\label{rem1}
	By the standard Fredholm theory, there exist
	\begin{itemize}
		\item a non decreasing sequence of nonegative real numbers $ \{\mu_k\}_{k=1}^\infty $ that tends to $ +\infty $ as $ k \rightarrow \infty $,
		\item a Hilbert basis $ \{\phi_k\}_{k=1}^\infty $ of $ L^2(\Omega) $ such that $ \phi_k \in H^1_0 (\Omega)$ such that
		\[
		\int_{\Omega} \nabla \phi_k \cdot \nabla \phi dx = \mu_k \int_\Omega \phi_k \phi dx \quad \text{for all } \phi \in H_0^1 (\Omega).
		\]
	\end{itemize}
	
\end{remark}

\begin{remark}\label{rem2}
	By Remark \ref{rem1}, we can take
	\begin{align}\label{choiceQ1}
	\mathbf{Q}_{\varepsilon}h=2\sum_{\mu_p \ge \frac{1}{2}\log(\gamma)}\mu_p\left\langle h,\phi_{p}\right\rangle \phi_{p}\quad \text{for }\gamma > 1.
	\end{align}
	It is immediate to see that the conditional one (\ref{QQ}) holds for $\mathbb{W}=\mathbb{G}_{1,1}(\Omega)$ and $C_0 = 2$ by using the Parseval identity. Here, cf. \cite{Tuan2017a}, we denote $\mathbb{G}_{\sigma,\alpha}(\Omega)$ by the Gevrey class of functions of order $\gamma > 0$ and index $\alpha>0$:
	\[
	\mathbb{G}_{\sigma,\alpha}\left(\Omega\right):=\left\{ u\in L^{2}\left(\Omega\right):\sum_{p=0}^{\infty}\mu_{p}^{\alpha}e^{2\sigma\mu_{p}}\left|\left\langle u,\phi_{p}\right\rangle \right|^{2}<\infty\right\} .
	\]
	By \eqref{choiceQ1} we obtain
	\begin{align}
	\mathbf{P}_{\varepsilon}h & =2\Delta h+\mathbf{Q}_{\varepsilon}h =-2\sum_{p\in\mathbb{N}}\mu_{p}\left\langle h,\phi_{p}\right\rangle \phi_{p}+2\sum_{\mu_p \ge \frac{1}{2}\log(\gamma)}\mu_p\left\langle h,\phi_{p}\right\rangle \phi_{p} \nonumber \\
	& =-2\sum_{\mu_p<\frac{1}{2}\log(\gamma)}\mu_{p}\left\langle h,\phi_{p}\right\rangle \phi_{p}. \label{choiceP1}
	\end{align}
	Thereupon, this $\mathbf{P}_{\varepsilon}$ satisfies the conditional estimate \eqref{PP} with $C_1 = 1$.
\end{remark}

\section{Well-posedness of the regularized system (\ref{eq:regu1})--(\ref{eq:regu2})} \label{sec:2}

Let $ v^\varepsilon (x,t) :=  e^{\rho (t-T)} u^\varepsilon (x,t)$ where $ \rho>1  $ is a constant chosen later.
Then~\eqref{eq:regu1}--\eqref{eq:regu2} become
\begin{equation}\label{eq:regu-v}
v_{tt}^\varepsilon
+ (1-2 \rho) v_t^\varepsilon 
+ (\rho^2-\rho) v^\varepsilon
+ (1-\rho) \Delta v^\varepsilon
+ \Delta v_t^\varepsilon
=  (1-\rho)\mathbf{P}_\varepsilon v^\varepsilon + \mathbf{P}_\varepsilon v^\varepsilon_t \; \text{ in } \Omega \times (0,T)
\end{equation}
and the boundary and terminal conditions:
\begin{equation}\label{eq:regu-bound-v}
\begin{cases}
v^\varepsilon (x,t) = 0 \quad\text{ on } \partial \Omega \times (0,T),\\
v^\varepsilon (x,T) =  f_0^\varepsilon (x) ,~v^\varepsilon_t (x,T)= \rho f_0^\varepsilon (x) + f_1^\varepsilon (x) \quad\text{ in } \Omega.
\end{cases}
\end{equation}

\begin{remark}
The most important difficult need to solve the regularized system~\eqref{eq:regu1}--\eqref{eq:regu2} lies in the term $ + \Delta u^\varepsilon $, which is bad for our energy estimations for  $ u^\varepsilon $. More precisely, the sign of this term is technically impeding the energy of the gradient term and eventually, it ruins our mathematical analysis in this section. In order to circumvent this, we consider the system~\eqref{eq:regu-v}--\eqref{eq:regu-bound-v} for $ v^\varepsilon $, which is equivalent to the regularized system~\eqref{eq:regu1}--\eqref{eq:regu2}.
Since $ \rho>1 $, then $ (1-\rho) \Delta v^\varepsilon$ becomes a ``good term'' and we shall use its effect to obtain the energy estimate for $ v^\varepsilon $ in Theorem~\ref{thm:0}. This leads us to the well-posedness of~\eqref{eq:regu-v}--\eqref{eq:regu-bound-v} as well as that of \eqref{eq:regu1}--\eqref{eq:regu2}. 	
\end{remark}

\begin{definition}\label{def:weak-v}
	
	A function $ v \in L^2 (0,T; H^1_0(\Omega))$ with $ v_t \in L^2 (0,T; H^1(\Omega)) $ and $ v_{tt} \in L^2 (0,T;H^{-1} (\Omega)) $ is a weak solution of~\eqref{eq:regu-v}--\eqref{eq:regu-bound-v} if for every test function $ \varphi \in H^1_0 (\Omega) $, it holds that
	\begin{align}\label{eq:weak-form}
	&\langle v_{tt} (t) , \varphi \rangle_{H^{-1}, H^1_0}
	+ (1-2\rho) \langle v_t (t),\varphi \rangle
	+ (\rho^2 - \rho) \langle v(t), \varphi \rangle	
	\\
&  	+ (\rho - 1 )\langle \nabla v (t), \nabla \varphi \rangle 
	- \langle \nabla v_t (t), \nabla \varphi \rangle 
	= (1-\rho) \langle \mathbf{P}_\varepsilon v (t), \varphi \rangle
	+ \langle \mathbf{P}_\varepsilon v_t(t), \varphi \rangle \nonumber
	\end{align}
	for a.e. $t\in (0,T)$, and $ v (x,T) = f_0^\varepsilon (x),~v_t  (x,T) =\rho f_0^\varepsilon (x) + f_1^\varepsilon (x) $ in $ \Omega $.
	
\end{definition}

Our proof of well-posedness relies on the conventional Galerkin method. This means that we construct solution of some finite-dimensional approximations to~\eqref{eq:weak-form}.

\begin{lemma}\label{lem:v_n}
	For any positive $ n $, there exist $ n $ absolutely continuous functions $ y_k^n :[0,T] \rightarrow \mathbb{R} $, $ k=1,\ldots,n $ and a function $ v_n \in L^2(0,T;H^1_0(\Omega)) $, where $ \partial_t v_n \in L^2(0,T;H^1(\Omega)) $ and $ \partial_{tt} v_n \in L^2(0,T;H^{-1} (\Omega)) $, of the form
	\begin{equation}\label{eq:v_n}
	v_n(x,t) = \sum_{k=1}^n y_k^n (t) \phi_k (x),
	\end{equation}
	such that for $ k=1,\ldots,n $
	\begin{equation}\label{eq:boundary-k}
	\begin{cases}
	y_k^n(T) = \int_\Omega f_0^\varepsilon (x) \phi_k (x) dx =: g_{0k} (T),\\
	\partial_t y_k^n(T) = \int_\Omega \left(\rho f_0^\varepsilon (x) + f_1^\varepsilon (x)\right) \phi_k  (x) dx =: g_{1k} (T),
	\end{cases}
	\end{equation}
	and $ v_n $ satisfies
	\begin{align}\label{eq:weak-form-k}
	&\int_\Omega \partial_{tt}  v_n(t) \phi_k dx
	+ (1-2\rho)  \int_\Omega \partial_t  v_n(t) \phi_k dx
	+ (\rho^2 - \rho) \int_\Omega v_n (t) \phi_k dx
	\nonumber \\
	&+ (\rho-1) \int_\Omega \nabla v_n(t) \cdot \nabla \phi_k dx
	-  \int_\Omega \nabla \partial_t  v_n(t) \cdot \nabla \phi_k dx \nonumber\\
	&= (1-\rho) \int_\Omega \mathbf{P}_\varepsilon v(t) \phi_k dx
	+ \int_\Omega \mathbf{P}_\varepsilon \partial_t v(t) \phi_k dx.
	\end{align}

\end{lemma}

\begin{proof}
By the properties of $ \{\phi_i\}_{i=1}^\infty $ in Remark~\ref{rem1}, \eqref{eq:weak-form-k} is equivalent to
\begin{align}\label{eq:01}
& \partial_{tt} y_k^n  (t)
+ (1-2\rho -\mu_k )\partial_t  y_k^n (t)
+ (\rho^2 +(\mu_k  - 1)\rho - \mu_k ) y_k^n (t)
\nonumber\\
&= (1-\rho) \sum_{i=0}^n y_i^n (t) \langle \mathbf{P}_\varepsilon \phi_i, \phi_k \rangle
+ \sum_{i=0}^n \partial_t  y_i^n (t) \langle \mathbf{P}_\varepsilon \phi_i, \phi_k \rangle \quad \text{for a.e. } t\in (0,T).
\end{align}	
Let $ z_k^n  = \frac{d}{dt} y_k^n$, it follows from~\eqref{eq:boundary-k} and~\eqref{eq:01} that
\begin{align*}
\dfrac{d}{dt}\left[\begin{array}{c}
y_{k}^{n}\\
z_{k}^{n}
\end{array}\right]
+A_k  \left[\begin{array}{c}
y_{k}^{n}\\
z_{k}^{n}
\end{array}\right]  = F_k, \quad \left[\begin{array}{c}
y_{k}^{n}(T)\\
z_{k}^{n}(T)
\end{array}\right]  =\left[\begin{array}{c}
g_{0k}(T)\\
g_{1k}(T)
\end{array}\right],
\end{align*}	
where $ F^n_k = [0,(1-\rho)\sum_{i=0}^{n}y_{k}^{n}\langle \mathbf{P}_{\varepsilon}\phi_{i},\phi_{k}\rangle+\sum_{i=0}^{n}z_{k}^{n}\langle \mathbf{P}_{\varepsilon}\phi_{i},\phi_{k}\rangle]^T $ and
\[
A_k = 
\left[\begin{array}{cc}
0 & 1\\
\rho^2 +(\mu_k - 1)\rho - \mu_k& 1 -2\rho- \mu_k
\end{array}\right] .
\]	
Consider $ w_k^n := [y_k^n, z_k^n]^T $. We thus obtain the following integral equation: 
	\begin{equation}\label{eq:02}
	w_k^n (t) = w_k^n (T) + A_k \int_t^T w_k^n (s) ds - \int_t^T F^n_k (s) ds.
	\end{equation}
Hereafter, we denote by $ w_n := [w_1^n, \ldots,w_n^n] : [0,T] \rightarrow \mathbb{R}^{2n}$.
	The integral equation~\eqref{eq:02} can be rewritten as $ w_n = H[w_n] $, where the same notation as $ w_n $ is applied to $ H $ with $ H_k^n $ being the right-hand side of~\eqref{eq:02}. To be more specific,
	\[
	H^n_k[w_n] (t) := w_k^n (T) + A_k \int_t^T w_k^n (s) ds - \int_t^T F^n_k (s) ds.
	\]
Define the norm in $ Y = C([0,T]; \mathbb{R}^{2n}) $ as follows:
\[
\left\Vert c \right\Vert_Y := \sup_{t\in [0,T]} \sum_{j=1}^n |c_j(t)| \quad \text{ with } c= [c_j] \in C([0,T]; \mathbb{R}^{2n}). 
\]
We claim that there exists $ n_0 \in \mathbb{N}^* $ such that the operator
\[ 
H^{(n_0)}:= H[H^{(n_0 -1) }] : Y \rightarrow Y
\]
 is a contraction mapping. In other words, we find $ K \in [0,1) $ such that
\[
\left\Vert H^{(n_0)} [w_n] -H^{(n_0)} [\tilde{w}_n] \right\Vert_Y 
\le 
K \left\Vert w_n - \tilde{w}_n \right\Vert_Y
\quad \text{for any } w_n, \tilde{w}_n \in Y.
\]
This can be done by induction. Indeed, let us observe that
\begin{align*}
& |H_k^n  [w_n] (t) - H_k^n  [\tilde{w}_n] (t) | \le 
\int_t^T |A_k| |w_k^n (s) - \tilde{w}_k^n (s)| ds \\
&+\int_t^T \left( C_1 C\log (\gamma)  \sum_{i=1}^{n} \left(  |1-\rho| |y_i^n (s) - \tilde{y}_i^n (s)|
+  |z_i^n (s) - \tilde{z}_i^n (s)| \right) \right)    ds  \\
& \le 
\int_t^T \left( |A_k| + C_1 C \log (\gamma) (\rho-1)  \right) |w_k^n (s) - \tilde{w}_k^n (s)|  ds \\
& \le 
\left(  |A_k| + C_1 C \log (\gamma) (\rho-1)  \right) (T-t) \left\Vert w_n - \tilde{w}_n \right\Vert_Y,
\end{align*}
aided by the conditional estimate~\eqref{PP}. Here, we indicate $ C = \max_i C( \left\Vert \phi_i \right\Vert_{H^1_0 (\Omega)}  )>0$.
Furthermore, for any $ m \in \mathbb{N}^* $
\begin{align*}
& | (H_k^n)^{(m)}  [w_n] (t) - (H_k^n)^{(m)}  [\tilde{w}_n] (t) | \\
& \le 
\int_t^T \left(  |A_k| + C_1 C \log (\gamma) (\rho-1)  \right)  | (H_k^n)^{(m -1)} [w_n] (s) - (H_k^n)^{(m-1)} [\tilde{w}_n] (s) |  ds ,
\end{align*}
and it follows by induction that
\begin{align*}
&| (H_k^n)^{(m)} [w_n] (t) - (H_k^n)^{(m)}  [\tilde{w}_n] (t) | \\
&\le \left(  |A_k| + C_1 C \log (\gamma) (\rho-1)  \right)^m \dfrac{(T-t)^m}{m!}
\left\Vert w_n - \tilde{w}_n \right\Vert_Y.
\end{align*}
Therefore, we obtain
\begin{align*}
& \left\Vert H^{(m)} [w_n] - H^{(m)} [\tilde{w}_n] \right\Vert_Y
\\
& \le \left\Vert w_n - \tilde{w}_n \right\Vert_Y \dfrac{ T^m}{m!} 
\sum_{k=1}^{n} \left(  | A_k |  + C_1 C \log (\gamma) (\rho-1)  \right)^m .
\end{align*}
Since the left-hand side tends to $ 0 $ as $ m \rightarrow \infty $,
we can find a sufficiently large $ m_0 $ such that
	\[
\dfrac{ T^{m_0}}{m_0!} 
\sum_{k=1}^{n} \left(  | A_k |  + C_1 C \log (\gamma) (\rho-1)  \right)^{m_0}  <1.
	\]
The claim is proved and by the Banach fixed-point argument, there exists a unique solution $ \bar{w}_n \in Y $ such that $ H^{(m_0)} [\bar{w}_n] = \bar{w}_n$.
Finally, since $ H^{(m_0)} \left[H[\bar{w}_n]\right] =  H \left[H^{(m_0)}[\bar{w}_n]\right] = H[\bar{w}_n]$, then the integral equation~\eqref{eq:02} admits a unique solution in $ Y $.
Hence, we complete the proof of the lemma.

\end{proof}

\begin{remark}
By Lemma~\ref{lem:v_n}, it is easy to check that there exists a constant $ C>0 $ such that
\begin{equation} \label{eq:06}
\left\Vert \partial_t v^\varepsilon_n(T) \right\Vert^2,
\left\Vert  v^\varepsilon_n(T) \right\Vert^2,
\left\Vert \nabla v^\varepsilon_n(T) \right\Vert^2 \le C \quad \text{for all } n \in \mathbb{N}.
\end{equation}
\end{remark}

\begin{theorem}\label{thm:0}
Assume~\eqref{eq:measure} holds.
	For each $ \varepsilon>0 $, the regularized system~\eqref{eq:regu-v}--\eqref{eq:regu-bound-v} admits a unique weak solution $ v^\varepsilon $ in the sense of Definition~\ref{def:weak-v}.
\end{theorem}
\begin{proof}
To prove the existence, we need to derive some energy estimates for approximate solution $ v_n^\varepsilon $.
	Thanks to Lemma~\ref{lem:v_n}, we have $ \partial_t v_n^\varepsilon  \in C([0,1]; H^1(\Omega))$.
	Multiplying~\eqref{eq:weak-form-k} by $ \partial_t y_k^n(t)  $, summing for $ k=1,\ldots,n $ and using the formula~\eqref{eq:v_n} for $ v_n^\varepsilon $, we get
	\begin{align*}
	&\int_\Omega \partial_{tt}  v^\varepsilon_n(t) \partial_t  v^\varepsilon_n(t)  dx
	+ (1-2\rho)  \int_\Omega | \partial_t v^\varepsilon_n(t)|^2  dx
	+ (\rho^2 - \rho) \int_\Omega v^\varepsilon_n (t)  \partial_t v^\varepsilon_n(t)  dx
	\nonumber \\
	&+ (\rho-1) \int_\Omega \nabla v^\varepsilon_n(t) \cdot \nabla  \partial_t v^\varepsilon_n(t)  dx
	-  \int_\Omega |\nabla \partial_t v^\varepsilon_n(t)|^2   dx\nonumber\\
	&= (1-\rho) \int_\Omega \mathbf{P}_\varepsilon (v^\varepsilon_n(t))  \partial_t v^\varepsilon_n(t)  dx
	+ \int_\Omega \mathbf{P}_\varepsilon (\partial_t v^\varepsilon_n(t))  \partial_t v^\varepsilon_n(t)  dx.
	\end{align*}
This implies
\begin{align}\label{eq:03}
&\dfrac{1}{2}\partial_t  \left[ \left\Vert \partial_t v^\varepsilon_n(t) \right\Vert^2
+ (\rho^2 - \rho) \left\Vert  v^\varepsilon_n(t) \right\Vert^2
+ (\rho-1) \left\Vert \nabla v^\varepsilon_n(t) \right\Vert^2
\right]
\nonumber\\
& 
- \left( 2\rho -1 \right)  \int_\Omega | \partial_t v^\varepsilon_n(t)|^2  dx
-  \int_\Omega |\nabla \partial_t v^\varepsilon_n(t)|^2 dx \nonumber\\
&= (1-\rho) \int_\Omega \mathbf{P}_\varepsilon (v^\varepsilon_n(t))  \partial_t v^\varepsilon_n(t)  dx
+ \int_\Omega \mathbf{P}_\varepsilon (\partial_t v^\varepsilon_n(t))  \partial_t v^\varepsilon_n(t) dx \nonumber\\
& \ge (1-\rho)  C_1 \log (\gamma) \left(\left\Vert v^\varepsilon_n(t) \right\Vert_{H^1(\Omega)}^2 
+
\left\Vert \partial_t v^\varepsilon_n(t) \right\Vert^2\right)
-  C_1 \log (\gamma)  \left\Vert \partial_t v^\varepsilon_n(t) \right\Vert^2,
\end{align}	
where the last inequality comes from the H\"older inequality and~\eqref{PP}.

\emph{Estimate $ v_n^\varepsilon $ in $ L^\infty (0,T;H^1(\Omega)) $ and $ \partial_t v_n^\varepsilon $ in  $ L^\infty (0,T;L^2(\Omega)) $}.
It follows from~\eqref{eq:03} that
\begin{align*}
&\partial_t  \left( \dfrac{\left\Vert \partial_t v^\varepsilon_n(t) \right\Vert^2}{\rho -1} 
+ \rho  \left\Vert  v^\varepsilon_n(t) \right\Vert^2
+ \left\Vert \nabla v^\varepsilon_n(t) \right\Vert^2
\right) 
\\
& \ge 
2C_1 \log (\gamma) \rho  \left(\dfrac{\left\Vert \partial_t v^\varepsilon_n(t) \right\Vert^2}{\rho -1} 
+ \rho  \left\Vert  v^\varepsilon_n(t) \right\Vert^2
+ \left\Vert \nabla v^\varepsilon_n(t) \right\Vert^2
\right),
\end{align*}
By Gr{\"o}nwall's inequality, we get
\begin{align}
&   \dfrac{\left\Vert \partial_t v^\varepsilon_n(t) \right\Vert^2}{\rho -1} 
+ \rho  \left\Vert  v^\varepsilon_n(t) \right\Vert^2
+ \left\Vert \nabla v^\varepsilon_n(t) \right\Vert^2
\nonumber\\
& \le 
\left(\dfrac{\left\Vert \partial_t v^\varepsilon_n(T) \right\Vert^2}{\rho -1} 
+ \rho  \left\Vert  v^\varepsilon_n(T) \right\Vert^2
+ \left\Vert \nabla v^\varepsilon_n(T) \right\Vert^2
\right)
\gamma^{2C_1 \rho (T-t)}. \label{eq:12}
\end{align}	
From~\eqref{eq:06}, one gets
\begin{equation}
\begin{cases}\label{eq:04}
 \partial_t v^\varepsilon_n  \text{ is uniformly bounded in } L^\infty (0,T; L^2 (\Omega)), \\
  v^\varepsilon_n  \text{ is uniformly bounded in } L^\infty (0,T; H^1 (\Omega)).
\end{cases}
\end{equation}
It follows from the Banach--Alaoglu theorem, and the argument that a weak limit of derivative is the derivative of the weak limit, that we can extract a subsequence of scaled approximate solutions $  v^\varepsilon_n $, which we still denote by $ \{v^\varepsilon_n \}_{n\in\mathbb{N}} $, such that for each $ \varepsilon>0 $
\begin{equation}\label{eq:11}
\begin{cases}
 \partial_t v^\varepsilon_n  \rightarrow \partial_t v^\varepsilon \text{ weakly}-* \text{ in } L^\infty (0,T; L^2 (\Omega)),\\
  v^\varepsilon_n  \rightarrow  v^\varepsilon \text{ weakly}-* \text{ in } L^\infty (0,T; H^1 (\Omega)). 
\end{cases}
\end{equation}

\emph{Estimate $ \partial_t v_n^\varepsilon $ in $ L^2(0,T;H^1_0(\Omega)) $}.
Integrating both sides of~\eqref{eq:03} from $ 0 $ to $ T $, we get
\begin{align*}
& (2\rho -1) \left\Vert \partial_t v_n^\varepsilon \right\Vert^2_{L^2(0,T;L^2(\Omega))} + \left\Vert \nabla v_n^\varepsilon\right\Vert^2_{L^2(0,T;L^2(\Omega))} \\
& \le C_1 \log(\gamma) \left( ( \rho -1 ) \left\Vert v_n^\varepsilon \right\Vert_{L^2(0,T;H^1_0(\Omega)}^2 + \rho \left\Vert \partial_t v_n^\varepsilon \right\Vert_{L^2(0,T;L^2(\Omega)}^2 \right) \\
& +\dfrac{1}{2} \left[ \left\Vert \partial_t v^\varepsilon_n(T) \right\Vert^2 
+ (\rho^2 - \rho) \left\Vert  v^\varepsilon_n(T) \right\Vert^2
+ (\rho-1) \left\Vert \nabla v^\varepsilon_n(T) \right\Vert^2
\right].
\end{align*}
From~\eqref{eq:06} and \eqref{eq:04}, it is straightforward to see that 
\begin{equation}\label{eq:07}
\left\Vert \partial_t v_n^\varepsilon \right\Vert_{L^2(0,T;H^1_0(\Omega))} \le \bar{C} \quad \text{for all }n \in \mathbb{N}.
\end{equation}
for some constant $ \bar{C} $.

\emph{Estimate $ \partial_{tt} v_n^\varepsilon$ in $ L^2(0,T;H^{-1}(\Omega)) $}.
Let $ \mathbb S_n $ be a closed subspace of $ H^1_0(\Omega) $ defined by $ \mathbb S_n = \{\varphi\in H^1_0(\Omega): \int_\Omega \varphi \varphi_k dx = 0 \text{ for all } k\le n \} $. Let $ \mathbb S^\perp_n $ be a closed subspace of $ H^1_0 (\Omega) $ such that $ H^1_0 (\Omega) = \mathbb S_n \oplus \mathbb S^\perp_n $.
In other words, for all $ \varphi \in H^1_0(\Omega)$, we can write $ \varphi $ of the form
$\varphi = \varphi_n + \varphi^\perp_n $ where $ \varphi \in \mathbb S_n  $ and $ \varphi^\perp_n \in  \mathbb S^\perp_n $.
Therefore, for  a.e. $ t\in [0,T] $, from~\eqref{eq:weak-form-k}, one gets
\begin{align*}
&	\langle \partial_{tt} v_n^\varepsilon(t), \varphi\rangle \\
&	= 
(2\rho -1 ) \langle \partial_t v_n^\varepsilon(t) , \varphi_n \rangle
+ (\rho - \rho^2)  \langle  v_n^\varepsilon(t) , \varphi_n \rangle
+ (1-\rho) \langle  \nabla v_n^\varepsilon(t), \nabla \varphi_n \rangle
\\
&   +\langle \nabla \partial_t v_n^\varepsilon, \nabla \varphi_n \rangle 
+ (1-\rho) \langle \mathbf{P}_\varepsilon (v_n^\varepsilon (t)), \varphi_n \rangle
+ \langle \mathbf{P}_\varepsilon (\partial_t v_n^\varepsilon (t)), \varphi_n \rangle \\
& \le (2\rho -1 ) \left\Vert \partial_t v_n^\varepsilon(t) \right\Vert \left\Vert \varphi_n \right\Vert
+ (\rho^2 - \rho) \left\Vert  v_n^\varepsilon(t) \right\Vert\left\Vert \varphi_n \right\Vert\\
&  + (\rho-1)  \left\Vert \nabla v_n^\varepsilon(t) \right\Vert \left\Vert \nabla \varphi_n \right\Vert
+ \left\Vert \partial_t \nabla v_n^\varepsilon(t) \right\Vert \left\Vert \nabla \varphi_n \right\Vert\\
& + (\rho-1) \left\Vert \mathbf{P}_\varepsilon (v_n^\varepsilon (t)) \right\Vert \left\Vert \varphi_n \right\Vert
+ \left\Vert \mathbf{P}_\varepsilon (\partial_t v_n^\varepsilon (t)) \right\Vert  \left\Vert \varphi_n \right\Vert.
\end{align*}
Since $  \left\Vert \varphi_n \right\Vert_{H^1_0(\Omega)}  \le  \left\Vert \varphi_n \right\Vert_{H^1_0(\Omega)}  + \left\Vert \varphi_n^{\perp} \right\Vert_{H^1_0(\Omega)}  = \left\Vert \varphi \right\Vert_{H^1_0(\Omega)}  $ for all $ n\in \mathbb{N} $,  we get
\begin{align*}
&\left\Vert\ \partial_{tt} v_n^\varepsilon (t) \right\Vert_{H^-1(\Omega)}  = \sup_{\varphi \in H^1_0(\Omega) \backslash \{0\} } \dfrac{\langle \partial_{tt} v_n^\varepsilon(t), \varphi\rangle}{\left\Vert \varphi \right\Vert_{H^1_0(\Omega)} }\\
& \le  (2\rho -1 ) \left\Vert \partial_t v_n^\varepsilon(t) \right\Vert
+ (\rho^2 - \rho) \left\Vert  v_n^\varepsilon(t) \right\Vert
+ (\rho-1)  \left\Vert \nabla v_n^\varepsilon(t) \right\Vert \\
&  + \left\Vert \partial_t \nabla v_n^\varepsilon(t) \right\Vert
+  C_1 \log(\gamma)  \left((1-\rho) \left\Vert  v_n^\varepsilon (t) \right\Vert_{H^1_0 (\Omega)} 
+ \left\Vert \partial_t v_n^\varepsilon (t) \right\Vert_{H^1_0 (\Omega)}  \right),
\end{align*}
where the last term in the right-hand side comes from the properties of $ \mathbf{P}_\varepsilon^1 $ and $ \mathbf{P}_\varepsilon^2 $.
From~\eqref{eq:04} and~\eqref{eq:07}, there exists a constant $ \tilde{C}>0 $ such that
\begin{equation}
\left\Vert \partial_{tt} v_n^\varepsilon \right\Vert_{L^2(0,T;H^{-1}(\Omega))} \le \tilde{C} \quad \text{for all }n \in \mathbb{N}.
\end{equation}
Henceforth, from the Banach--Alaoglu theorem, there exists a subsequence of $ \{v_n^\varepsilon\} $ (still denoted by $ \{v_n^\varepsilon\} $) such that
\begin{equation}\label{eq:13}
\partial_{tt} v_n^\varepsilon \rightarrow \partial_{tt} v^\varepsilon \text{ weakly in } L^2(0,T;H^{-1}(\Omega)).
\end{equation}
Combining the above weak-star and weak limits, the function $ v^\varepsilon $ satisfies
\begin{equation*}
\begin{cases}
v^\varepsilon \in L^\infty (0,T; H^1_0(\Omega)),\\
\partial_t v^\varepsilon \in L^\infty (0,T;L^2(\Omega)) \cap L^2(0,T;H^1_0 (\Omega)),\\
\partial_{tt} v^\varepsilon \in L^2(0,T;H^{-1} (\Omega)).
\end{cases}
\end{equation*}
Furthermore, since $ H_0^1 (\Omega)$ is compactly embedded in $ L^2(\Omega) $ and $ L^2(\Omega) $ is continuously embedded in $ H^{-1} (\Omega) $ (by Rellich--Kondrachov), from Aubin--Lions lemma, we get
\begin{equation}\label{eq:08}
\begin{cases}
v_n^\varepsilon \rightarrow v^\varepsilon \text{ strongly in } C([0,T]; H_0^1(\Omega)), \\
\partial_t v_n^\varepsilon \rightarrow  \partial_t v^\varepsilon \text{ strongly in } C([0,T]; L^2(\Omega)).
\end{cases}
\end{equation}
Fix an integer $ N $ and choose a function $ \bar{v} \in C^1(0,T; H^1_0(\Omega ))$ having the form
\begin{equation}\label{eq:09}
\bar{v} (t) = \sum_{k=1}^{N} d_k (t) \phi_k,
\end{equation}
where $ d_1, \ldots,d_N $ are given real valued $ C^1 $ functions defined in $ [0,T] $.
For all $ x \ge N $, multiplying~\eqref{eq:weak-form-k}, summing for $ k=1,\ldots,N $  and integrating over $ (0,T) $ lead to
\begin{align*}
&\int_\Omega \partial_{tt}  v_n^\varepsilon(t) \bar{v} dx
+ (1-2\rho)  \int_\Omega \partial_t v_n^\varepsilon(t) \bar{v} dx
+ (\rho^2 - \rho) \int_\Omega v_n^\varepsilon (t) \bar{v} dx
\nonumber \\
&+ (\rho-1) \int_\Omega \nabla v_n^\varepsilon(t) \cdot \nabla \bar{v} dx
-  \int_\Omega \nabla \partial_t v_n^\varepsilon(t) \cdot \nabla \bar{v} dx \nonumber\\
&= (1-\rho) \int_\Omega \mathbf{P}_\varepsilon v_n^\varepsilon(t) \bar{v} dx
+ \int_\Omega \mathbf{P}_\varepsilon \partial_t v_n^\varepsilon(t) \bar{v} dx.
\end{align*}
Letting $ n\rightarrow \infty $, we obtain from~\eqref{eq:08} that
\begin{align}\label{eq:10}
&\int_\Omega \partial_{tt}  v^\varepsilon(t) \bar{v} dx
+ (1-2\rho)  \int_\Omega \partial_t v^\varepsilon(t) \bar{v} dx
+ (\rho^2 - \rho) \int_\Omega v^\varepsilon (t) \bar{v} dx
\nonumber \\
&+ (\rho-1) \int_\Omega \nabla v^\varepsilon(t) \cdot \nabla \bar{v} dx
-  \int_\Omega \nabla \partial_t v^\varepsilon(t) \cdot \nabla \bar{v} dx \nonumber\\
&= (1-\rho) \int_\Omega \mathbf{P}_\varepsilon v^\varepsilon(t) \bar{v} dx
+ \int_\Omega \mathbf{P}_\varepsilon \partial_tv^\varepsilon(t) \bar{v} dx.
\end{align}
Since the functions of the form~\eqref{eq:09} are dense in $ L^2(0,T;H^1_0(\Omega)) $, the equality~\eqref{eq:10} holds for all test function $ \bar{v} \in L^2(0,T;H^1_0 (\Omega)) $.
We deduce that the function $ v^\varepsilon $ obtained from approximate solutions $ v_n^\varepsilon $ satisfies the weak formulation in Definition~\ref{def:weak-v}.

It now remains to verify the initial data for $ v^\varepsilon $.
Take $ \kappa \in C^1([0,T]) $ satisfying $ \kappa(T) = 1 $ and $ \kappa(0) = 0 $.
It follows from~\eqref{eq:11} that
\[
\int_0^T \langle \partial_t v_n^\varepsilon (t),  \phi \rangle \kappa (t) dt \rightarrow \int_0^T \langle \partial_t v^\varepsilon (t),  \phi \rangle \kappa (t) dt \quad \text{for all } \phi \in H^1_0(\Omega).
\]	
Then by integration by parts, one gets
\begin{align*}
 \int_0^T \langle  v_n^\varepsilon (t),  \phi \rangle \partial_t \kappa (t) dt 
& - \langle  v_n^\varepsilon (T),  \phi \rangle  \kappa (T) \\
& \rightarrow 
 \int_0^T \langle  v^\varepsilon (t),  \phi \rangle \partial_t \kappa (t) dt 
- \langle  v^\varepsilon (T),  \phi \rangle  \kappa (T) 
\end{align*}	 
and thereupon, we get $ \langle  v_n^\varepsilon (T),  \phi \rangle \rightarrow \langle  v^\varepsilon (T),  \phi \rangle $	for all $ \phi \in H^1_0(\Omega) $ by virtue of~\eqref{eq:11}.
From Lemma~\ref{lem:v_n}, we also have that $ v_n^\varepsilon (T) \rightarrow f_0^\varepsilon $ in $ L^2(\Omega) $ as $ n \rightarrow \infty $.
Thus  $ \langle  v^\varepsilon (T),  \phi \rangle = \langle  f^\varepsilon_0,  \phi \rangle $ for all $ \phi \in H^1_0 (\Omega) $, which implies that $ v^\varepsilon(T) = f^\varepsilon_0 $ a.e. in $ \Omega $.
Similarly, 
it follows from~\eqref{eq:13} that
\[
\int_0^T \langle \partial_{tt} v_n^\varepsilon (t),  \phi \rangle \kappa (t) dt \rightarrow \int_0^T \langle \partial_{tt} v^\varepsilon (t),  \phi \rangle \kappa (t) dt \quad \text{for all } \phi \in H^1_0(\Omega).
\]	
Then by integration by parts, one gets
\begin{align*}
&- \int_0^T \langle  \partial_t v_n^\varepsilon (t),  \phi \rangle \partial_t \kappa (t) dt 
+ \langle \partial_t v_n^\varepsilon (T),  \phi \rangle  \kappa (T) \\
& \rightarrow - \int_0^T \langle \partial_t v^\varepsilon (t),  \phi \rangle \partial_t \kappa (t) dt 
+ \langle \partial_t v^\varepsilon (T),  \phi \rangle  \kappa (T) \quad\text{ as } n \rightarrow \infty .
\end{align*}
Using the similar arguments as in the proof for $ v^\varepsilon (T) $, we  obtain that $ \partial_t v^\varepsilon(T) = \rho f^\varepsilon_0 + f^\varepsilon_1$ a.e. in $ \Omega $	. Hence, we complete the proof of the existence.
%
%
%

Finally, we are going to prove the uniqueness of \eqref{eq:regu-v}--\eqref{eq:regu-bound-v}.
We sketch out some important steps because this proof is standard.
Indeed, let $ v^\varepsilon $ and $ \bar{v}^\varepsilon $ be two weak solutions of the system~\eqref{eq:regu-v}--\eqref{eq:regu-bound-v}.	
Since the system is linear, it is straightforward to see that the function $ k^\varepsilon = v^\varepsilon - \bar{v}^\varepsilon  $	satisfies~\eqref{eq:regu-v} with zero terminal conditions $ k^\varepsilon (T) = \partial_t k^\varepsilon (T) = 0 $.
Taking $ \varphi = \partial_t k^\varepsilon $ as a test function, we proceed as in the way to get the estimate~\eqref{eq:12}.
Hence, $ k^\varepsilon(t) =0 $ a.e. in $ (0,T) $ because of the fact that
\[
 \dfrac{\left\Vert \partial_t k^\varepsilon(t) \right\Vert^2}{\rho -1} 
 + \rho  \left\Vert  k^\varepsilon (t) \right\Vert^2
 + \left\Vert \nabla k^\varepsilon (t) \right\Vert^2 \le 0 \quad \text{a.e. in } (0,T). 
\]
This completes the proof of the theorem. 
\end{proof}

\section{Convergence analysis}\label{sec:4}
In this part, our focus is on the convergence analysis of the variational QR framework adapted to solve the time-reversed hyperbolic heat conduction problem. The error estimate obtained below can be viewed as a ``worst-case'' scenario of convergence of this QR scheme in case the stabilized operator $\mathbf{P}_{\varepsilon}$ is bounded logarithmically.

It is worth noting that our analysis in section \ref{sec:2} does not care about the dependence of $C$ (and any type of constants in there) on the noise level $\varepsilon$, since basically we fix $\varepsilon$. However, to this end any constant $C>0$ used below should be $\varepsilon$-independent because we are going to show the error estimates with respect to only $\varepsilon$.

\begin{theorem}\label{thm:1}
		Assume (\ref{eq:measure}) holds. Let $\varepsilon\in\left(0,1\right)$ be a sufficiently small number
		such that $\gamma:=\gamma\left(\varepsilon\right)\ge e^{2/C_{1}}$. Suppose the
		following conditions hold 
		\begin{align}\label{eq:assum}
		\begin{cases}
		3C_{1}T<2,\\
		\lim_{\varepsilon\to0}\gamma^{2}\left(\varepsilon\right)\varepsilon\le K.
		\end{cases}
		\end{align}
		Next, assume the original system \eqref{eq:ori1}\textendash \eqref{eq:ori2} admits a unique solution $u$ such that $u\in C([0,T];\mathbb{W})$ and $u_t \in L^2(0,T;\mathbb{W})$, where $\mathbb{W}$ is obtained in Definition \ref{def1}. Let $M>0$ be such that
		\[
		\left\Vert u\right\Vert _{C\left(\left[0,T\right];\mathbb{W}\right)}^{2}+\left\Vert u_{t}\right\Vert _{L^2\left(0,T;\mathbb{W}\right)}^{2}\le M.
		\]
		Let $u^{\varepsilon}$ be a unique weak solution of the regularized system \eqref{eq:regu1}\textendash \eqref{eq:regu2} analyzed in Theorem~\ref{thm:0}. Then for $0\le t \le T$ the following error estimates hold:
		\begin{align*}
			& \left\Vert u^{\varepsilon}\left(t\right)-u\left(t\right)\right\Vert ^{2}\le C\left(\varepsilon+\left(\log(\gamma)\right)^{-1}\gamma^{3C_{1}\left(T-t\right)-2}\right),\\
			& \left\Vert \nabla u^{\varepsilon}\left(t\right)-\nabla u\left(t\right)\right\Vert ^{2}\le C\left(\log\left(\gamma\right)\varepsilon+\gamma^{3C_{1}\left(T-t\right)-2}\right),\\
			& \left\Vert u_{t}^{\varepsilon}\left(t\right)-u_{t}\left(t\right)\right\Vert ^{2}+\int_{t}^{T}\left\Vert \nabla u_{t}^{\varepsilon}\left(s\right)-\nabla u_{t}\left(s\right)\right\Vert ^{2}ds
			\\ &
			\le C\left(\left(\log(\gamma)\right)^2\varepsilon
			+\log\left(\gamma\right)\gamma^{3C_{1}\left(T-t\right)-2}\right).
		\end{align*}
		where $C=C\left(K,M,C_{0},C_{1}\right)>0$ is independent of $\varepsilon$.
\end{theorem}
\begin{proof}
	Let $w^{\varepsilon}\left(x,t\right)=\left[u^{\varepsilon}\left(x,t\right)-u\left(x,t\right)\right]e^{\rho_{\varepsilon}\left(t-T\right)}$
	for some $\rho_{\varepsilon}>0$, viewing as a weighted difference
	function in our proof of convergence. The notion behind this use of
	the Carleman weight function is to ``maximize'' the measured terminal
	data that we are having and thus, we can take full advantage of the
	noise level $\varepsilon$. The weight function here is classical in the framework of parabolic equations backward in time; cf. e.g. \cite[Section 9]{Yamamoto2009}. In principle, the downscaling (with respect
	to the noise level) used here is helpful in getting rid of the large
	stability magnitude by a suitable choice of the auxiliary parameter
	$\rho_{\varepsilon}$, which is also relatively large. Now, we compute
	the equation for $w^{\varepsilon}$, calling as the difference equation
	between the regularized problem \eqref{eq:regu1}\textendash \eqref{eq:regu2}
	and the original system \eqref{eq:ori1}\textendash \eqref{eq:ori2}.
	In fact, we have
	\begin{align}
		 w_{t}^{\varepsilon}&=\left[u_{t}^{\varepsilon}-u_{t}\right]e^{\rho_{\varepsilon}\left(t-T\right)}+\rho_{\varepsilon}\left[u^{\varepsilon}-u\right]e^{\rho_{\varepsilon}\left(t-T\right)} \nonumber\\
		& 
		=\left[u_{t}^{\varepsilon}-u_{t}\right]e^{\rho_{\varepsilon}\left(t-T\right)}+\rho_{\varepsilon}w^{\varepsilon},\label{eq:5a}\\
		 \Delta w^{\varepsilon}&=\left[\Delta u^{\varepsilon}-\Delta u\right]e^{\rho_{\varepsilon}\left(t-T\right)},
	\end{align}
	which lead to
	\begin{align}
		 w_{tt}^{\varepsilon}-\rho_{\varepsilon}w_{t}^{\varepsilon}&=\left[u_{tt}^{\varepsilon}-u_{tt}\right]e^{\rho_{\varepsilon}\left(t-T\right)}+\rho_{\varepsilon}\left[u_{t}^{\varepsilon}-u_{t}\right]e^{\rho_{\varepsilon}\left(t-T\right)} \nonumber\\
		 &=\left[u_{tt}^{\varepsilon}-u_{tt}\right]e^{\rho_{\varepsilon}\left(t-T\right)}+\rho_{\varepsilon}\left(w_{t}^{\varepsilon}-\rho_{\varepsilon}w^{\varepsilon}\right),\\
		 \Delta w_{t}^{\varepsilon}-\rho_{\varepsilon}\Delta w^{\varepsilon}&=\left[\Delta u_{t}^{\varepsilon}-\Delta u_{t}\right]e^{\rho_{\varepsilon}\left(t-T\right)}.\label{eq:8a}
	\end{align}
	Hereby, we notice that when multiplying both sides of the systems
	\eqref{eq:regu1}\textendash \eqref{eq:regu2} and \eqref{eq:ori1}\textendash \eqref{eq:ori2}
	by the weight $e^{\rho_{\varepsilon}\left(t-T\right)}$, it yields
	\begin{align}
		& \left[u_{tt}^{\varepsilon}-u_{tt}\right]e^{\rho_{\varepsilon}\left(t-T\right)}+\left[u_{t}^{\varepsilon}-u_{t}\right]e^{\rho_{\varepsilon}\left(t-T\right)}
		+\Delta\left(u^{\varepsilon}-u\right)e^{\rho_{\varepsilon}\left(t-T\right)} \nonumber \\ &+\Delta\left(u_{t}^{\varepsilon}-u_{t}\right)e^{\rho_{\varepsilon}\left(t-T\right)} =\mathbf{P}_{\varepsilon}\left(u^{\varepsilon}-u\right)e^{\rho_{\varepsilon}\left(t-T\right)}+\mathbf{Q}_{\varepsilon}ue^{\rho_{\varepsilon}\left(t-T\right)} \nonumber
		\\ &+\mathbf{P}_{\varepsilon}\left(u_{t}^{\varepsilon}-u_{t}\right)e^{\rho_{\varepsilon}\left(t-T\right)}+\mathbf{Q}_{\varepsilon}u_{t}e^{\rho_{\varepsilon}\left(t-T\right)}.
		\label{eq:9a}
	\end{align}
	Henceforth, we plug the identities \eqref{eq:5a}\textendash \eqref{eq:8a}
	into the equation \eqref{eq:9a} to get
	\begin{align}
		& w_{tt}^{\varepsilon}+\left(\rho_{\varepsilon}^{2}-\rho_{\varepsilon}\right)w^{\varepsilon}-\left(\rho_{\varepsilon}-1\right)\Delta w^{\varepsilon}+\Delta w_{t}^{\varepsilon}\nonumber \\
		& =\mathbf{P}_{\varepsilon}w^{\varepsilon}+\mathbf{Q}_{\varepsilon}ue^{\rho_{\varepsilon}\left(t-T\right)}+\left(2\rho_{\varepsilon}-1\right)w_{t}^{\varepsilon}+\mathbf{P}_{\varepsilon}\left(w_{t}^{\varepsilon}-\rho_{\varepsilon}w^{\varepsilon}\right)+\mathbf{Q}_{\varepsilon}u_{t}e^{\rho_{\varepsilon}\left(t-T\right)}.\label{eq:PDEw}
	\end{align}
	which is the PDE for the difference function $w^{\varepsilon}$.
	
	Now, we multiply both sides of \eqref{eq:PDEw} by $w_{t}^{\varepsilon}$
	and integrate the resulting equation over $\Omega$. After some manipulations,
	we arrive at
	\begin{align}
		& \frac{1}{2}\frac{d}{dt}\left\Vert w_{t}^{\varepsilon}\right\Vert ^{2}+\frac{1}{2}\left(\rho_{\varepsilon}^{2}-\rho_{\varepsilon}\right)\frac{d}{dt}\left\Vert w^{\varepsilon}\right\Vert ^{2}+\frac{1}{2}\left(\rho_{\varepsilon}-1\right)\frac{d}{dt}\left\Vert \nabla w^{\varepsilon}\right\Vert ^{2}-\left\Vert \nabla w_{t}^{\varepsilon}\right\Vert ^{2}\nonumber \\
		& =\left(2\rho_{\varepsilon}-1\right)\left\Vert w_{t}^{\varepsilon}\right\Vert ^{2}+\left\langle \left(\mathbf{P}_{\varepsilon}-\rho_{\varepsilon}\mathbf{P}_{\varepsilon}\right)w^{\varepsilon},w_{t}^{\varepsilon}\right\rangle +\left\langle \mathbf{P}_{\varepsilon}w_{t}^{\varepsilon},w_{t}^{\varepsilon}\right\rangle \nonumber \\ & +e^{\rho_{\varepsilon}\left(t-T\right)}\left\langle \mathbf{Q}_{\varepsilon}u,w_{t}^{\varepsilon}\right\rangle +e^{\rho_{\varepsilon}\left(t-T\right)}\left\langle \mathbf{Q}_{\varepsilon}u_{t},w_{t}^{\varepsilon}\right\rangle .\label{eq:11a}
	\end{align}
	Based upon the conditional estimates (\ref{QQ})--(\ref{PP}) we estimate the right-hand side
	of \eqref{eq:11a} as follows:
	\begin{align}
		& \left\langle \left(\mathbf{P}_{\varepsilon}-\rho_{\varepsilon}\mathbf{P}_{\varepsilon}\right)w^{\varepsilon},w_{t}^{\varepsilon}\right\rangle \ge-\frac{1}{2}\left(\rho_{\varepsilon}-1\right)C_{1}^{2}\left(\log(\gamma)\right)^2\left\Vert w^{\varepsilon}\right\Vert ^{2}-\frac{1}{2}\left(\rho_{\varepsilon}-1\right)\left\Vert w_{t}^{\varepsilon}\right\Vert ^{2},\nonumber \\
		& \left\langle \mathbf{P}_{\varepsilon}w_{t}^{\varepsilon},w_{t}^{\varepsilon}\right\rangle \ge-C_{1}\log\left(\gamma\right)\left\Vert w_{t}^{\varepsilon}\right\Vert ^{2}, \label{a3} \\
		& e^{\rho_{\varepsilon}\left(t-T\right)}\left\langle \mathbf{Q}_{\varepsilon}u,w_{t}^{\varepsilon}\right\rangle 
		\ge
		-\frac{1}{2}\left(\dfrac{1}{4}\left\Vert w_{t}^{\varepsilon}\right\Vert ^{2}+4e^{2\rho_{\varepsilon}\left(t-T\right)}C_{0}^2\gamma^{-2}\left\Vert u\right\Vert ^{2}_{\mathbb{W}_{1}}\right), \label{a1}\\
		& e^{\rho_{\varepsilon}\left(t-T\right)}\left\langle \mathbf{Q}_{\varepsilon}u_{t},w_{t}^{\varepsilon}\right\rangle 
		\ge
		-\frac{1}{2}\left(\dfrac{1}{4}\left\Vert w_{t}^{\varepsilon}\right\Vert ^{2}+4e^{2\rho_{\varepsilon}\left(t-T\right)}C_{0}^2\gamma^{-2}\left\Vert u_{t}\right\Vert ^{2}_{\mathbb{W}_{2}}\right). \label{a2}
	\end{align}
	Therefore, by integrating (\ref{eq:11a}) from $t$ to $T$ we estimate that
	\begin{align*}
		& \left\Vert w^\varepsilon_{t}\left(t\right)\right\Vert ^{2}+(\rho_{\varepsilon}^{2}-\rho_{\varepsilon})\left\Vert w^{\varepsilon}\left(t\right)\right\Vert ^{2}+(\rho_{\varepsilon}-1)\left\Vert \nabla w^{\varepsilon}\left(t\right)\right\Vert ^{2}+2\int_{t}^{T}\left\Vert \nabla w_{t}^{\varepsilon}\left(s\right)\right\Vert ^{2}ds\\
		& \le\left\Vert w^\varepsilon_{t}\left(T\right)\right\Vert ^{2}+(\rho_{\varepsilon}^{2}-\rho_{\varepsilon})\left\Vert w^{\varepsilon}\left(T\right)\right\Vert ^{2}+\left(\rho_{\varepsilon}-1\right)\left\Vert \nabla w^{\varepsilon}\left(T\right)\right\Vert ^{2}\\
		& +4C_{0}^2\gamma^{-2}\rho_{\varepsilon}^{-1}\left(1-e^{2\rho_{\varepsilon}\left(t-T\right)}\right)\left\Vert u\right\Vert _{C\left(\left[0,T\right];\mathbb{W}_1\right)}^{2} 
		+4C_{0}^2\gamma^{-2}\left\Vert u_{t}\right\Vert _{L^2\left(0,T;\mathbb{W}_2\right)}^{2}\\
		& +C_{1}^{2}\rho_{\varepsilon}^{-1}\left(\log(\gamma)\right)^{2}\int_{t}^{T}\rho_{\varepsilon}\left(\rho_{\varepsilon}-1\right)\left\Vert w^{\varepsilon}\left(s\right)\right\Vert ^{2}ds\\
		& +2\left[\frac{1}{2}\left(\rho_{\varepsilon}-1\right)+C_{1}\log\left(\gamma\right)+ \dfrac{1}{2}-2\rho_{\varepsilon}+1\right]\int_{t}^{T}\left\Vert w_{t}^{\varepsilon}\left(s\right)\right\Vert ^{2}ds.
	\end{align*}
	By choosing $\rho_{\varepsilon}=C_{1}\log\left(\gamma\right)\ge2$
	(since $\gamma\ge e^{2/C_{1}}$), the last term in the right-hand side becomes 
	$(2-\rho_\varepsilon)\int_t^T  \left\Vert w_{t}^{\varepsilon}\left(s\right)\right\Vert ^{2}ds \le 0 $,
	we apply the Gr\"onwall
	inequality to obtain
	\begin{align}
		& \left\Vert w_{t}^{\varepsilon}\left(t\right)\right\Vert ^{2}+\left(\rho_{\varepsilon}^{2}-\rho_{\varepsilon}\right)\left\Vert w^{\varepsilon}\left(t\right)\right\Vert ^{2}+\left(\rho_{\varepsilon}-1\right)\left\Vert \nabla w^{\varepsilon}\left(t\right)\right\Vert ^{2}+2\int_{t}^{T}\left\Vert \nabla w_{t}^{\varepsilon}\left(s\right)\right\Vert ^{2}ds\nonumber \\
		& \le\left[2\left(\rho_{\varepsilon}^{2}+1\right)\varepsilon^{2}+\varepsilon^{2}\left(\rho_{\varepsilon}^{2}-1\right) + 4C_0^2 \gamma^{-2}M\right]\gamma^{C_{1}\left(T-t\right)}
		,\label{eq:13a}
	\end{align}
	where we have used the measurement assumption \eqref{eq:measure}
	and the fact that
	\begin{equation}\label{haha}
	\left\Vert w^{\varepsilon}_{t}\left(T\right)\right\Vert ^{2}
	=\left\Vert \left[u_{t}^{\varepsilon}\left(T\right)-u_{t}\left(T\right)\right] +\rho_{\varepsilon}\left[u^{\varepsilon}\left(T\right)-u\left(T\right)\right]\right\Vert ^{2} 
\le2\left(\rho_{\varepsilon}^{2}+1\right) \varepsilon^{2}.
	\end{equation}
	Thus, using the back-substitution
	\begin{equation}\label{eq:14}
	w^{\varepsilon}\left(x,t\right)=\left[u^{\varepsilon}\left(x,t\right)-u\left(x,t\right)\right]e^{\rho_{\varepsilon}\left(t-T\right)}=\left[u^{\varepsilon}\left(x,t\right)-u\left(x,t\right)\right]\gamma^{C_{1}\left(t-T\right)},
	\end{equation}
	we conclude the convergence in $L^{2}\left(\Omega\right)$ type as
	follows:
	\begin{align*}
		& \left\Vert u^{\varepsilon}\left( t\right)-u\left(t\right)\right\Vert ^{2} \\ & \le\left(\frac{2\left(\rho_{\varepsilon}^{2}+1\right)}{\rho_{\varepsilon}^{2}-\rho_{\varepsilon}}+\frac{\rho_{\varepsilon}^{2}-1}{\rho_{\varepsilon}^{2}-\rho_{\varepsilon}}\right)\varepsilon^{2}\gamma^{3C_{1}\left(T-t\right)} +\frac{4}{\rho_{\varepsilon}^{2}-\rho_{\varepsilon}}C_{0}^2M\gamma^{-2}\gamma^{3C_{1}\left(T-t\right)}\\
		& \le
		\frac{\rho_{\varepsilon}^{2}-1}{\rho_{\varepsilon}^{2}-\rho_{\varepsilon}}\left(3\gamma^{3C_{1}\left(T-t\right)}+\gamma^{3C_{1}\left(T-t\right)}\right)\varepsilon^{2}  +4C_{0}^2M\rho_{\varepsilon}^{-1}\gamma^{3C_{1}\left(T-t\right)-2}\\
		& \le2\left(4\gamma^{3C_{1}\left(T-t\right)}\varepsilon^{2}+C_{1}^{-1}2C_{0}^2M\left(\log(\gamma)\right)^{-1}\gamma^{3C_{1}\left(T-t\right)-2}\right).
	\end{align*}
From~\eqref{eq:assum}, we get $ \gamma^{3C_{1}\left(T-t\right)}\varepsilon^{2} \le K^{\frac{3C_1T}{2}} \varepsilon $ and it follows from the previous inequality that
\begin{equation}\label{eq:15}
\left\Vert u^{\varepsilon}\left( t\right)-u\left(t\right)\right\Vert ^{2} 
\le
C \left(\varepsilon + \left(\log(\gamma)\right)^{-1}\gamma^{3C_{1}\left(T-t\right)-2}\right),
\end{equation}
for some constant $ C >0$.
	In the same manner, we derive from \eqref{eq:13a} the convergence
	for the gradient terms:
	\begin{align*}
		 \left\Vert \nabla u^{\varepsilon}\left(t\right)-\nabla u\left(t\right)\right\Vert ^{2}
		& \le2\left(4C_{1}\log\left(\gamma\right)\gamma^{3C_{1}\left(T-t\right)}\varepsilon^{2}
		+2C_{0}M\gamma^{3C_{1}\left(T-t\right)-2}\right)\\
		& \le C\left(\log\left(\gamma\right)\varepsilon+\gamma^{3C_{1}\left(T-t\right)-2}\right).
	\end{align*} 
Now using the back-substitution~\eqref{eq:14}, we get
\[
\nabla w_t^\varepsilon (t) = \left[ \nabla u_t^\varepsilon (t) - \nabla u_t(t) \right] \gamma^{C_1 (t-T)} + \rho_\varepsilon \left[ \nabla u^\varepsilon (t) - \nabla u(t) \right] \gamma^{C_1 (t-T)}. 
\]
It yields
\begin{align*}
& 2\int_t^T \left\Vert\nabla w_t^\varepsilon (s) \right\Vert^2 ds
+  2\int_t^T \left\Vert  \nabla u^\varepsilon (s) - \nabla u(s) \right\Vert^2 \rho_\varepsilon^2\gamma^{2C_1 (s-T)} ds\\ 
&\ge  \int_t^T \left\Vert  \nabla u_t^\varepsilon (s) - \nabla u_t(s) \right\Vert^2 \gamma^{2C_1 (s-T)} ds \\ & \ge \gamma^{2C_1 (t-T)} \int_t^T \left\Vert  \nabla u_t^\varepsilon (s) - \nabla u_t(s) \right\Vert^2  ds.
\end{align*}
Thus it follows from~\eqref{eq:13a} that
\begin{align*}
& \int_t^T \left\Vert  \nabla u_t^\varepsilon (s) - \nabla u_t(s) \right\Vert^2 ds\\ 
\le &
\left (4\rho_{\varepsilon}^{2}\varepsilon^{2}+ 4C_0^2 \gamma^{-2}M\right)\gamma^{3C_{1}\left(T-t\right)}
+ 2 \rho_\varepsilon^2  \gamma^{2C_1 (t-T)} \int_t^T \left\Vert  \nabla u^\varepsilon (s) - \nabla u(s) \right\Vert^2 ds\\
 \le &
\left[ 4 C_1^2  (\log(\gamma))^2 \varepsilon^2  \gamma^{3C_{1}\left(T-t\right)} +  4C_0^2 M \gamma^{3C_{1}\left(T-t\right) -2 }\right] \\
& + C T \rho_\varepsilon^2  \gamma^{2C_1 (t-T)} \left(\varepsilon + \left(\log(\gamma)\right)^{-1}\gamma^{3C_{1} T-2}\right),
\end{align*}
where we have used the estimate~\eqref{eq:15} for the last inequality.
This implies
\[
\int_t^T \left\Vert  \nabla u_t^\varepsilon (s) - \nabla u_t(s) \right\Vert^2 ds
\le C\left(\left(\log(\gamma)\right)^{2}\varepsilon+\log\left(\gamma\right)\gamma^{3C_{1}\left(T-t\right)-2}\right).
\]
Finally, using the back-substitution~\eqref{eq:14}, one has	
\[
w_t^\varepsilon (t) = \left[  u_t^\varepsilon (t) -  u_t(t) \right] \gamma^{C_1 (t-T)} + \rho_\varepsilon \left[  u^\varepsilon (t) -  u(t) \right] \gamma^{C_1 (t-T)}. 
\]	
This implies
\begin{align*}
\left\Vert  u_t^\varepsilon (t) -  u_t(t) \right\Vert^2 \gamma^{2C_1 (t-T)}
\le
2\left\Vert w_t^\varepsilon (s) \right\Vert^2
+ 2 \rho_\varepsilon^2  \left\Vert  u^\varepsilon (t) -  u(t) \right\Vert^2.
\end{align*}
Applying the estimate of~$ \left\Vert w_t^\varepsilon\right\Vert^2 $ in \eqref{eq:13a} and $ \left\Vert u^{\varepsilon}\left( t\right)-u\left(t\right)\right\Vert ^{2}  $ in~\eqref{eq:15}, we obtain
\[
\left\Vert   u_t^\varepsilon (t) -  u_t(t) \right\Vert^2 
\le C\left((\log\left(\gamma\right))^{2}\varepsilon+\log\left(\gamma\right)\gamma^{3C_{1}\left(T-t\right)-2}\right).
\] 
	Hence, we complete the proof of the theorem.	

\end{proof}

As a by-product of Theorem \ref{thm:1}, an appropriate choice of $\gamma$ is taken to state the following convergence result with the H\"older rates. It is then noticeable that our error estimates in Theorem \ref{thm:1} and in Corollary \ref{corol:1} below are uniform in time. In this regard, they are still true for $t=0$ under the restriction \eqref{eq:assum}. This H\"older convergence result is quite different from that of the backward heat equations, which is usually logarithmic at $t=0$ and is of H\"older type at $t>0$; cf. e.g. \cite{Hao2009, Hao2018, Long1994}. In the framework of the backward hyperbolic problems, our convergence result can also be compared with those obtained in \cite{Tuan2017b}.

\begin{corollary}\label{corol:1}
	Under the assumptions of Theorem \ref{thm:1}, if we choose $\gamma\left(\varepsilon\right) = \varepsilon^{-1/2}$, then for any $\varepsilon \le e^{-4/C_1}$ the following error estimates hold:
	\begin{align*}
		& \left\Vert u^{\varepsilon}\left(t\right)-u\left(t\right)\right\Vert ^{2}\le C\left(\varepsilon+(\log(\varepsilon^{-1/2}))^{-1}\varepsilon^{1-3C_{1}\left(T-t\right)/2}\right),\\
		& \left\Vert \nabla u^{\varepsilon}\left(t\right)-\nabla u\left(t\right)\right\Vert ^{2}\le C\left(\log(\varepsilon^{-1/2})\varepsilon+\varepsilon^{1-3C_{1}\left(T-t\right)/2}\right),\\
		& \left\Vert u_{t}^{\varepsilon}\left(t\right)-u_{t}\left(t\right)\right\Vert ^{2}+\int_{t}^{T}\left\Vert \nabla u_{t}^{\varepsilon}\left(s\right)-\nabla u_{t}\left(s\right)\right\Vert ^{2}ds
		\\ &
		\le C\left((\log(\varepsilon^{-1/2}))^{2}\varepsilon+\log(\varepsilon^{-1/2})\varepsilon^{1-3C_{1}\left(T-t\right)/2}\right).
	\end{align*}
	where $C=C\left(M,C_{0},C_{1}\right)>0$ is independent of $\varepsilon$.
\end{corollary}


\begin{remark}
	\begin{itemize}
		\item If we apply the perturbing and stabilized operators \eqref{choiceQ1} and \eqref{choiceP1} in Remark \ref{rem2} to Theorem \ref{thm:1}, regularity of the true solution of the original system \eqref{eq:ori1}\textendash \eqref{eq:ori2} is restricted in the Gevrey space. More precisely, one has $u\in C([0,T];\mathbb{W}_{1,1})$ and $u_{t}\in C([0,T];\mathbb{W}_{1,1})$ in Theorem \ref{thm:1}. The same result is applied for Corollary \ref{corol:1}.
		
		\item Obviously, the perturbation for $-\Delta u$ and $-\Delta u_t$ of \eqref{eq:ori1} can be different from each other. This will also lead to different regularity assumptions on the exact solution that we have assumed in Theorem \ref{thm:1}.
		
		\item We remark that if the measurement assumption (\ref{eq:measure}) is only given by
		\[
		\left\Vert u^{\varepsilon}\left(\cdot,T\right)-u\left(\cdot,T\right)\right\Vert \le \varepsilon,
		\]
		we obtain the logarithmic rate of convergence in the following sense:
		\[
		\left\Vert u^{\varepsilon}\left(t\right)-u\left(t\right)\right\Vert ^{2}\le C/(\log(\gamma))^2.
		\]
	\end{itemize}
\end{remark}

\bibliography{sample}                                             %

%
%
%
%
%
%

\end{document}